\documentclass[a4paper,10pt]{article}
\usepackage{graphicx}
\usepackage[utf8]{inputenc}
\usepackage{amsthm,amsmath, amssymb,amsfonts}
\usepackage[all]{xy}
\usepackage[Sonny]{fncychap}
\usepackage{fancyhdr}
\usepackage{latexsym}
\usepackage{geometry}
\usepackage[pdftex]{hyperref}
\usepackage{mathpazo}
\usepackage[all]{xy}
\title{Sequential Gibbs Measures and Factor Maps}
\author{Giovane Ferreira\footnote{Ferreira is supported by CAPES, INCTMat and FAPEAL. This work is part of the PhD thesis at UFAL/UFBA.} \; and Krerley Oliveira \footnote{ Oliveira is partially supported by CNPq, CAPES, FAPEAL, INCTMAT and Foundation Louis D.}}
\newtheorem{theorem}{Theorem}
\newtheorem{cor}{Corollary}
\newtheorem{corollary}{Corollary}[section]
\newtheorem{proposition}{Proposition}[section]
\newtheorem{definition}{Definition}[section]
\newtheorem{remark}{Remark}[section]
\newtheorem{lemma}{Lemma}[section]
\newtheorem{example}{Example}[section]

\newcommand{\ud}{\underline}
\newcommand{\ov}{\overline}
\newcommand{\bb}{\mathbb}

\begin{document}
 \maketitle

\begin{abstract}
We define the notion of \emph{sequential Gibbs measures}, inspired by on the classical notion of Gibbs measures and recent examples from the study of non-uniform hyperbolic dynamics. Extending previous results of Kempton-Pollicott \cite{Kempton-Pollicot}  and Ugalde-Chazottes \cite{chazottes.ugalde2}, we show that the images of one block factor maps of a sequential Gibbs measure are also a sequential 
Gibbs measure, with the same sequence of Gibbs times. We obtain some estimates on the regularity of the potential of the image measure at almost every point.  
\end{abstract}

\section{Introduction}

Introduced in the early seventies in the realm of Dynamical Systems, Gibbs measures plays an important role in the understanding of the Ergodic Theory of hyperbolic  or expanding maps. These measures are equilibrium states of expanding maps for regular potentials. However, even the existence of such measures requires strong forms of regularity and hyperbolicity. This makes it difficult to make use such measures beyond the uniformly hyperbolic dynamics and creates the need of adapt and extend this concept in the non-uniform hyperbolic setting. 

Trying to understand the dynamics of intermittent maps and study its equilibrium states, Yuri generalized Gibbs measures defining the notion of weak Gibbs measure, where uniform control of the measure of dynamic balls at every point by a constant  is replaced by a  subexponential sequence of constants. Compare with Definition~\ref{def.yuri} and see more in \cite{yuri_1999,yuri2,yuri3}.

  More recently, with the rapid growth of the study and understanding of non-uniformly hyperbolic maps,  several works were carried out in the context of non-uniformly expanding dynamics dealing with  more general measures inspired by Gibbs measures, such as non-lacunary Gibbs measures.  See \cite{OV,VV,ramosviana}, just to refer  some of them. These measures are equilibrium states for some non-uniformly expanding maps and potentials and their Gibbs-like property is even weaker than analogous property of weak Gibbs measures in the sense of Yuri. The subexponential sequence of constants is replaced by a subexponential sequence of functions defined almost everywhere. It means that the non-uniform control at  every point is replaced by a non-uniform control at \emph{almost} every point, as was present in (\cite{OV}, Proposition 3.17).

Here, we study the behavior of Gibbs-like properties under factor maps.  To describe this problem precisely, let us consider two full shifts spaces $\Sigma_i = \{1,\dots,k_i\}^\mathbb{N}$, for $i=1,2$, and a surjective map $\pi:\{1,\dots,k_1\} \rightarrow \{1,\dots,k_2\}$ and   extend $\pi$ to a surjective map  $\Pi: \Sigma_1 \rightarrow \Sigma_2$, defining 
$$
\Pi(x_{1}x_{2}...)=\pi(x_{1})\pi(x_{2})....
$$

 This map is called an \emph{one block} factor map.

Given a continuous potential $\phi$ on a full shift, one can not expect that Gibbs measures exists or they are unique, as was shown by Hofbauer[\cite{Hofbauer}, page 230]. The regularity of the potential plays a important role in the existence and uniqueness of Gibbs and equilibrium measures.  To analyze it in detail,  consider the \emph{$n$-variation}  of $\phi$ defined by
$$
var_{n}(\phi)=\sup\{|\phi(\underline{z})-\phi(\underline{w})|: w_i=z_i, \text{ for } i=0,\dots,n-1\}.
$$
The uniform continuity of the function $\phi$ corresponds to $var_{n}(\phi)\rightarrow 0$ as $n\rightarrow\infty$ and the H\"older continuity of $\phi$ corresponds to the existence of  constants $C>0$ and $\theta \in (0,1)$ such that $var_{n}(\phi)< C\, \theta^{n}$, for $n \geq 1$.

Given a Gibbs measure $\mu$ on $\Sigma_1$ for a continuous function $\psi_1$, we consider its image $\nu:=\Pi_*\mu$ under $\Pi$. One interesting question  that arises in the Theory of Hidden Markov Chains is to show that $\nu$ is also a Gibbs measure for some continuous function $\psi_2$. In the case where $\mu$ is a Markov measure, sufficient conditions for $\nu$ to be a Gibbs
measure were given in \cite{chazottes.ugalde1,yoo2010}. The case where $\mu$ is a Gibbs measure and $\Sigma_1$ is a full shift,  we list some important recent  contributions: 

%fazer as referencias corretas abaixo
\begin{itemize}   

\item  (\cite{Kempton-Pollicot}, Theorem 1.1 and \cite{chazottes.ugalde2}, Theorem 3.1) If $\psi_1$ continuous, then $\nu$ is a Gibbs measure for some continuous potential $\psi_2$.

\item (\cite{redig-wang},Theorem 2) If $\psi_1$ H\"older continuous, then $\nu$ is a Gibbs measure for some H\"older continuous potential $\psi_2$.

 We say that $\psi$ is \emph{stretched} H\"older, if there are constants $t,C>0$ and $\theta \in (0,1)$ such that $var_{n}(\psi)< C\, \theta^{n^t}$, for $n \geq 1$. %See Section~\ref{}. 

\item (\cite{Kempton-Pollicot}, Theorem 5.3 and \cite{chazottes.ugalde2}, Theorem 4.1)  If $\psi_1$ is stretched H\"older, then  $\psi_2$ can be chosen stretched H\"older. 

\item (\cite{Kempton-Pollicot}, Theorem 5.1) If $\sum_{n\geq 1} n^{d+1} var_n(\psi_1) <+\infty$  for some $d\geq 0$, then  $\psi_2$ is such that $\sum_{n\geq 1} n^{d} var_n(\psi_2) <+\infty.$

\item (\cite{kempton},Theorem 1.8 and \cite{piraiano}, Theorem 2) Similar results for the case of subshift of finite type under fiber-wise mixing assumption. 

%\item{}  If $\psi_1$ H\"older continuous, then $\nu$ is a Gibbs measure for some H\"older continuous potential $\psi_2$, in the case of subshift of finite type under fiber-wise mixing assumption. 

\item (\cite{yayama},Theorem 3.1) A non-additive version of the same problem: if $\mu$ is a Gibbs measure for a sequence of  almost additives potentials $\Psi=\{\psi_n\}_n$ on $\Sigma_1$ with bounded variation, then $\Pi_* \mu$ is a Gibbs measure for a sequence of continuous potentials $\Phi=\{\phi_n\}_n$. 

\end{itemize}

The main object of this paper is the notion of \emph{sequential} Gibbs measures. We say that $\mu$(not necessarily invariant probability) is a sequential Gibbs measure with respect to $\phi: \Sigma \rightarrow \mathbb{R}$ if there are  constants $K,P$ such that for $\mu$-a.e. $\underline{x} \in \Sigma$, there exists an increasing sequence of natural numbers $n_i(\underline{x})\in \mathbb{N}$ such that for every $0\leq j \leq n_i-1$ 

\begin{equation}\label{eq.1gibbs}
 K^{-1}\leq \frac{\mu\big( [x_j\dots x_{n_i-1}])}{e^{\phi^{n_i-j}(\ud x)-(n_i-j)P}}\leq K,
\end{equation} 
where $\phi^n(\ud x) = \sum_{i=0}^{n-1} \phi(f^i(\ud x))$ and $ [x_0x_1\cdots x_{n-1}] = \{\ud y=y_0y_1\dots \in \Sigma; \text{ such that } y_i=x_i, \text{ for } i=0,\dots,n-1\}.$  If Equation~\eqref{eq.1gibbs} holds for the sequence $n_i(\ud x)=i$ and every $\ud x \in \Sigma$, the sequential Gibbs measures is just a standard Gibbs measure.

The maximal subsequence $n_i(\ud x)$ satisfying the Equation~\eqref{eq.1gibbs} is called the sequence of \emph{Gibbs times} of $\ud x$.  Note that if $n_{i}(\underline{x})$ is a Gibbs time of $\underline{x}$ then for $n\leq n_{i}(\ud x)$, $n_{i}(\ud x)-n$ is Gibbs time of $\sigma^{n}(\underline{x})$.

 We give natural examples of sequential Gibbs measures that are not standard Gibbs measures in Section~\ref{examples}, where we discuss equilibrium states on shifts constructed in \cite{Hofbauer} and image of non-lacunary Gibbs measures of local diffeomorphisms for  H\"older potentials studied in \cite{OV}, \cite{VV} and \cite{ramosviana} under coding by some partition.

The results that we obtain here are a kind of non-uniform counterparts of those in \cite{Kempton-Pollicot,chazottes.ugalde2} adapted for sequential Gibbs measures, for much less regular potentials and more suitable for the study of non-uniformly hyperbolic dynamical systems. We prove that given a  sequential Gibbs measure $\mu$ for a continuous potential $\psi_{1}$ on a shift $\Sigma_1$ and  $\Pi: \Sigma_1 \rightarrow \Sigma_2$ a  one block factor map that is  regular with respect to $\mu$, then  the measure $\nu:= \Pi_*\mu$ on $\Sigma_2$ is a  sequential Gibbs measure for some almost everywhere continuous potential  $\psi_{2}:\Sigma_{2} \rightarrow \mathbb{R}$. We also obtain local estimates almost everywhere for the regularity of $\psi_2$ based on the regularity almost everywhere of $\psi_1$. 

The results obtained in this paper can be extended to the case of subshift of type finite with the property of topologically mixing in the fibers, as in  \cite{kempton}.  We also expect that the theorems obtained here would be useful to obtain a Central Limit Theorem for pointwise dimension of non-lacunary Gibbs measures, using the approach of  \cite{leplaideursausool}. 
%\footnote{R2. comentar sobre generalizações no ttrbalho do Kempton. Mas isso foi feito!!}

\section{Results}\label{s.results}

%\footnote{R2...pergunta sobre o conjunto G. É denso? entropia total? Fiz um exemplo abaixo baseado em Ramos-Viana dando condições para a pressão de G ser total}
The definition of sequential Gibbs measures depends on the constants $K$ and $P$.  However, if we denote by $G=\{\ud x \in \Sigma;  \text{$\ud x$ has infinitely many Gibbs times}\}$, we prove that $P$ is uniquely determined by the \emph{pressure} $P_G(\psi)$ of $\psi$ with respect to $G$. 

To define  $P_G(\psi)$, denote by $\mathcal{C}_n$ the set of all cylinders of length $n$. We consider the family $m_\alpha(\cdot,\psi,N)$   of exterior measures defined by 
$$
m_\alpha(G,\psi,N)= \inf_{\mathcal{U}}  \left\{\sum_{C \in \mathcal{U}} e^{-\alpha n(C) + \sup_{x\in C} \psi^{n(C)}(x)}\right\},
$$ where the infimum is take over all open covers $\mathcal{U} \subset \cup _{n\geq N}\mathcal{C}_n$ of $G$ and $n(C)$ is the length of $C$. Then, we set  
$$
m_\alpha(G,\psi)= \lim_{N\rightarrow \infty} m_\alpha(G,\psi,N) 
$$ and define
$$
P_G(\psi)=\inf\{\alpha \in \mathbb{R}; m_\alpha(G,\psi)=0\}.
$$ For more details and properties  about $P_G(\psi)$, we suggest \cite[Section 11, Chapter 4]{Pesin}. Now, we prove that

% An interesting question is when the topological pressure of G is full. The example below gives us an answer to this.

%\begin{example}
%	If $h_{top}(G^c)<h_{top}(G)$, then if $\psi:\Sigma \rightarrow \mathbb{R}$ satisfies $\sup_{x\in G^c}\{\psi(x)\}\leq \inf_{x\in G}\{\psi(x)\}$ we have
	
%	\begin{eqnarray*}
	%	P_{G^c}(\psi)&\leq&  h_{top}(G^c) +\sup_{\ud x \in G^c} \psi(x)
%		< h_{top}(G) +\inf_{\ud x \in G} \psi(x)\leq P_G(\psi)
%	\end{eqnarray*}	
%	Then, $P_G(\psi)=P_{\Sigma}(\psi)$.\end{example}

\begin{proposition}\label{p.pressaoG} If $\psi$ admits a sequential Gibbs measure $\mu$, then  $P=P_G(\psi)$ is the unique number that satisfy Equation~\eqref{eq.1gibbs}, where $G$ is the set of points with infinitely many Gibbs times. If $\mu$ is an ergodic invariant measure, $P_G(\psi) = h_\mu(\sigma)+\int \psi \,d\mu$. 
\end{proposition}

\begin{proof}

In fact, assume that $\mu$ is a sequential Gibbs measure with constants $K$ and $P$ satisfying Equation~\eqref{eq.1gibbs}. For the first part, denote by $\mathcal{G}_n$ the collection of all cylinders $C=[x_0\dots x_{n-1}]$ such that $n$ is a Gibbs time of some $\ud x \in C$.  Fixed $k$, by definition  of  $G$ we have that  $\mathcal{U}_k= \cup_{n>k} \mathcal{G}_{n}$ is an open cover of $G$ and  
$$
\mathcal{V}_k = \bigcup\limits_{n>k}\{[x_0\dots x_{n-1}] \in \mathcal{G}_n; [x_0\dots x_{l-1}] \notin \mathcal{G}_l, \text{ for } k\leq l< n\}
$$ is an open partition  of $G$.  For any $\gamma>P$ we have that
\begin{equation*}
\begin{aligned}
m_\gamma(G,\psi,k) \leq \sum_{C \in \mathcal{V}_k} e^{-\gamma n(C) + \sup_{\ud x\in C} \psi^{n(C)}(\ud x)} &=& \\
 =   \sum_{C \in \mathcal{V}_k}e^{-(\gamma-P)n(C)} e^{-Pn(C) + \sup_{\ud x\in C} \psi^{n(C)}(\ud x)}  &\leq& Ke^{-(\gamma-P)k} \sum_{C \in \mathcal{V}_k} \mu(C).  
\end{aligned}
\end{equation*} Since $\sum_{C \in \mathcal{V}_k} \mu(C) \leq 1$, taking   $k \rightarrow \infty$, we have that $m_\gamma(G,\psi,k)=0$ and   $P\geq P_G(\psi)$. The opposite inequality follows in a similar fashion from the fact that for every  cylinder $C=[x_0\dots x_{n-1}]$ such that $n$ is a Gibbs time of some point $\underline{x}$ of $C$ we have that $ e^{-P n +  \psi^n(\underline{x})} \geq K^{-1} \mu(C).$ 

Now, we prove that $P_G(\psi) = h_\mu(\sigma)+\int \psi \,d\mu$ for an ergodic invariant sequential Gibbs measure. To finish the proof, just observe that by Brin-Katok's local entropy formula we have that for almost every  $\underline{x} \in G$, if $n_i(\ud x)$ is the sequence of Gibbs times of $\ud x$, then   
$$
h_\mu(\sigma) = -\lim \frac{1}{n_i}\log\mu([x_0\dots x_{n_i-1}]) = P_G(\psi)-\int \psi\,d\mu.
$$  

\end{proof}

By the previous Proposition, since $P$ is uniquely defined, we call it the \emph{pressure} of the sequential Gibbs measure $\mu$.  Through this paper, we assume that the constant $K$ in Equation~\eqref{eq.1gibbs} is fixed. Without loss of generality,  we  assume that  $P=0$ in Equation~\eqref{eq.1gibbs}, since  $\mu$ is a sequential Gibbs  measure  for $\psi$ with pressure $P$  if, and only if, $\mu$ is a sequential Gibbs measure for  $\psi-P$ with pressure zero.

\begin{remark}\label{r.G}
In some examples, the potential $\psi$ and the sequential Gibbs measure $\mu$ are obtained from an equilibrium state of a H\"older potential and a  nonuniformly expanding map, using a semiconjugacy with a full shift. The set $G$ is the preimage under the semiconjugacy of the set of points with infinitely many hyperbolic times. See Example~\ref{ex.localdiffeos} for details.

Under some hypothesis about the potential, like small variation condition, the set $G$ has full pressure, i.e., every measure with big pressure gives full measure to the set $G$. This is discussed with more detail at Example \ref{ex.localdiffeos}. In the Example \ref{ex.hofbauer}, we present another situation where $G$ is a dense set with full entropy.
\end{remark}

\begin{definition}\label{fs hyp}
%	\footnote{R2. GS é invariante aqui? é de probabilidade? Enfatizar}
 We say that an one block factor map $\Pi:\Sigma_1 \rightarrow \Sigma_2$ is regular with respect to a sequential Gibbs measure $\mu$ on $\Sigma_1$, if there exists a $\mu$-full measure set $D \subset G \subset \Sigma_1$, such that  given  $\underline{x} \in D$ then  $\Pi^{-1}(\Pi(\ud x)) \subset G$  and  $n_{1}(\underline{ x})=n_1(\underline{y})$, for every $\ud{y} \in \Pi^{-1}(\Pi(\ud x))$. 
 
\end{definition}

From now on, we consider only probability measures $\mu$  such that $\mu(\sigma^{-1}(A))=0$ for every set $A\subset \Sigma_1$ with $\mu(A)=0$. Since $n_k(\ud x)=n_1(\sigma^{n_{k-1}(\ud x)}(\ud x))$,  if $\Pi$ is regular with respect $\mu$  then we may define $E= \Pi(\cap_{k\geq 0} \sigma^{-k}(D))$ and observe that $E\subset \Sigma_2$ is a $\nu$-full measure set such that    given  $\underline{x}, \underline{y}\in \Pi^{-1}(E)$, with  $\Pi(\underline{x})=\Pi(\underline{y})$ then  $n_{k}(\underline{x})=n_k(\underline{y}),$ for every $k\geq 1$.

We define the $n$-th  variation  of $\phi$ at $\ud x=x_0x_1...x_n... $  as
\[var_{n}(\phi,\ud x)=\sup \{|\phi(\underline{x})-\phi(\underline{w})|: \ud w\in [x_{0}...x_{n-1}] \}.\]
where  $[x_{0}...x_{n-1}]=\{\underline{w}\in \Sigma:w_{0}...w_{n-1}=x_{0}...x_{n-1}\}$ is the cylinder of length $n$ at $\ud x$.

We define the variation of a potential $\phi:\Sigma \rightarrow \mathbb{R}$ on the set $K \subset \Sigma$ with respect to $X$ by
\[var_n(\phi,K):=\sup \limits_{\ud x \in K}var_n(\phi,\ud x).\]

Now, we state  the first result of this paper:

\begin{theorem}\label{fs.teo proj}
Let  $\mu$  be a sequential Gibbs measure for a continuous  potential $\psi_{1}:\Sigma_1 \rightarrow \mathbb{R}$. If  $\Pi$ is regular with respect to $\mu$, then  the measure $\nu:=\Pi_*\mu$ on $\Sigma_2$ is a  sequential Gibbs measure for some  potential  $\psi_{2}:\Sigma_{2} \rightarrow \mathbb{R}$, continuous at $\nu$ almost every point. 
\end{theorem}

  Now we study the modulus of continuity of  $\psi_2$ at some point with respect to the modulus of  $\psi_1$ at its preimages. 

%\footnote{R1. GS no Teo 1 e 2 são não-									singulares?}
\begin{theorem}\label{fs.teo reg1}

 Let  $\mu$  be a sequential Gibbs measure for a continuous potential $\psi_{1}:\Sigma_1 \rightarrow \mathbb{R}$. If for $\nu$-a.e. $\ud z\in \Sigma_2$, we have that  $\limsup n_k(z)/k <+\infty$ and  there exist a decreasing positive  function $f_{\ud z} : \mathbb{N} \rightarrow \mathbb{R}$ such that $\limsup f_{\ud z}(k)k < +\infty$   and for every $1\leq j \leq n_k$ we have    
$$
var_{j}(\psi_1,\Pi^{-1}(\sigma^{n_k-j}(\ud z)))< f_{\ud z}(j).
$$

Then, given any $\gamma<1$ there are constants $0<\alpha<1$ and $C>0$  such that for $\nu$-almost every point $\ud z \in \Sigma_2$,  there exists $k_0(\ud z)$ such that for each $k>k_0(\ud z)$ given $ \ud z' \in [z_0,\dots,z_k]$, then 
  $$
	|\psi_2(\ud z)-\psi_2(\ud z')|< C\max\{ \alpha^{k^{1-\gamma}}, f_{\ud z}([k^\gamma])k  \}.
$$

\end{theorem}

It follows directly from the Theorem~\ref{fs.teo reg1} that:

\begin{cor}
	(Local stretched Hölder decay): Suppose that there are constants $\Gamma_1>0$  and 
	$\beta_1, \theta_{1} \in (0,1)$ such that for almost every $\ud{ w} \in E$, if  $n_k(\ud{w})$ is the sequence of Gibbs times of $\ud{w}$ and $1\leq j \leq n_k$, we have that for  $k$ big enough
	$$
	var_{j}(\psi_1,\Pi^{-1}(\sigma^{n_k-j}(\ud w))< \Gamma_1{\theta_{1}}^{j^{\beta_1}}.
	$$
	
	Then, we may choose $\psi_2$ in such way that $\nu$ is a sequential Gibbs measure for $\psi_2$ and  there are constants $\Gamma_2>0$  and 	$\beta_2, \theta_{2} \in (0,1)$ such that for almost every $\ud w \in \Sigma_2$, there exists $k_0(\ud w)$ such that given $ \ud w' \in [w_0,\dots,w_k]$ then 
  $$
	|\psi_2(\ud w)-\psi_2(\ud w')|< \Gamma_2 {\theta_{2}}^{k^{\beta_2}},
	$$ where  $\theta_{2}=\max{\alpha,\theta_1} \in (0,1)$  and $\Gamma_2>0$. 
\end{cor}
	
	\begin{proof}
		In Theorem \ref{fs.teo reg1} we put $f_{\ud w}(j)=\Gamma_1\cdot{\theta_{1}}^{j^{\beta_1}}$. Then, for $\gamma<1$, we have that for $ \ud w' \in [w_0,\dots,w_k]$ then 
  $$
	|\psi_2(\ud w)-\psi_2(\ud w')|< C\max\{ \alpha^{k^{1-\gamma}}, \Gamma_1{\theta_{1}}^{[k^\gamma]^{\beta_1}}k \}
		$$
		
		Let $k_0$ such that $\theta:=\theta_1 k_ 0^{\frac{1}{[k_0^\gamma]^{\beta_1}}}<1$ and $\beta<1$ such that $[k^{\gamma}]^{\beta_1}\geq k^{\beta}$. Then, for $k\geq k_0$, we have
		$$
		|\psi_2(\ud w)-\psi_2(\ud w')|< C\max\{ \alpha^{k^{1-\gamma}}, \Gamma_1{\theta}^{ \beta} \}< \Gamma_2 \theta_2^{k^{\beta_2}}
		$$
		where $\theta_2=\max\{\alpha,\theta\}$ and $\beta_2=\max\{1-\gamma, \beta\}$.
	\end{proof}

\begin{cor}
	(Local polynomial decay): Suppose that there constants $\Gamma_1>0$  and 
	$r>2$ such that for every $\ud{ w} \in E$, if  $n_k(\ud{w})$ is the sequence of Gibbs times of $\ud{w}$, we have that  for any $1\leq j \leq n_k$
	$$
	var_{j}(\psi_1,\Pi^{-1}(\sigma^{n_k-j}(\ud w)))< \Gamma_1j^{-r}.
	$$ 
	
	Then, we may choose $\psi_2$ in such way that $\nu$ is a sequential Gibbs measure for $\psi_2$ and there are constants $\Gamma_2>0$ and for every $s<r-1$  such that for $\nu$-a.e. there exists $k_0(\ud w)$ such that given $ \ud w' \in [w_0,\dots,w_k]$ then 
  $$
	|\psi_2(\ud w)-\psi_2(\ud w')|< \Gamma_2 k^{-s}.
	$$ 
\end{cor}
	
	\begin{proof}
		In Theorem \ref{fs.teo reg1} we put $f_{\ud w}(k)=\Gamma_1 k^{-r}$. For $\nu$-a.e. there exists $k_0(\ud w)$ such that given $ \ud w' \in [w_0,\dots,w_k]$ then 
		
		$$
		|\psi_2(\ud w)-\psi_2(\ud w')|< C\max\{ \alpha^{k^{1-\gamma}}, \Gamma_1 [k^\gamma]^{-r}k \}\leq C\,\Gamma_1 [k^\gamma]^{-r}k
		$$
		We can choose $\lambda<1$ such that $[k^{\gamma}]>k^{\lambda}$ for every $k$. Indeed,
		
		\[k^{\gamma \cdot \lambda}-1\leq \frac{k^{\gamma}-1}{k^{\lambda}}\leq \frac{[k^{\gamma}]}{k^{\lambda}}\]
		Then
		\[
		|\psi_2(\ud w)-\psi_2(\ud w')|< C\,\Gamma_1 k^{1-\lambda r}< C\,\Gamma_1k^{-s}\]
		To finish the proof, just put $\Gamma_2=C\,\Gamma_1$.

	\end{proof}

\begin{cor}
	(Local summable variations)
	Suppose that for any $\ud{ w} \in E$, if  $n_k$ is a Gibbs time of $\ud w$, we have that  	
	$$
	\sum_{k\geq 1} k                                                                                                                                                                                                                                                  var_{k}(\psi_1,\Pi^{-1}(\ud w))<\infty 
	$$ 
	Then, we may choose $\psi_2$ in such way that $\nu$ is a sequential Gibbs measure for $\psi_2$ and   such that for $\nu$-a.e. $\ud w \in \Sigma_2$ there exists $k_0(\ud w)$ such that $\ud w' \in [w_0...w_k]$ then
	 $$\sum\limits_{k>k_0(\ud w)}  {k} |\psi_2(\ud w)-\psi_2(\ud w')|< \infty.$$ 
\end{cor}
	\begin{proof}
		Let $\beta>0$ and $\gamma=\frac{1}{\beta+1}$, as in proof of the Theorem \ref{fs.teo reg1}. Define $f_{\ud w}(k)=var_{[k^{\frac{1}{\gamma}}]+1}(\psi_1,\Pi^{-1}(\ud w))$. For $\nu$-a.e. there exists $k_0(\ud w)$ such that given $ \ud w' \in [w_0,\dots,w_k]$
		\begin{eqnarray*}
			|\psi_2(\ud w)-\psi_2(\ud w')|&<& C\max\{ \alpha^{k^{1-\gamma}},k\, var_{[[k^{\gamma}]^{\frac{1}{\gamma}}]+1}(\psi_1,\Pi^{-1}(\ud w)) \}\\
			&\leq&
			C\max\{ \alpha^{k^{1-\gamma}},k\, var_{k}(\psi_1,\Pi^{-1}(\ud w))\}
		\end{eqnarray*}
		Obviously, $\sum\limits_{k\geq k_0(\ud w)}  {k}\alpha^{k^{1-\gamma}}<\infty$. Then, jointly with the hypothesis, we have
		\[\sum\limits_{k>k_0(\ud w)}  k|\psi_2(\ud w)-\psi_2(\ud w')|< \infty.\]
	\end{proof}

Before start the proofs of Theorem~\ref{fs.teo proj} and Theorem~\ref{fs.teo reg1}, we prove some properties of sequential Gibbs measures. First,  we observe that the first Gibbs time function give us some information about the growth of the function $n_k$. In fact, define the function $n_1: G \rightarrow \mathbb{R}$ the \emph{first Gibbs time} of $\ud x$. Then,

\begin{proposition} If $\mu$ is an ergodic sequential Gibbs measure such that $\int n_1 \,d\mu<+\infty$,  then for $\mu$-a.e. $\ud x \in \Sigma_1$, there exists $b(\ud x)$ such that for every $k\geq 0$ we have that $n_k(\ud x)\leq  b k$. 
\end{proposition}

\begin{proof}
 Let $G$ be the set of points with infinitely many Gibbs times. We may define  $g:G \to G$ by
\begin{equation*}\label{eq.transf_induzida}
 \quad g(\ud x) = \sigma^{n_1(\ud x)}(\ud x),
\end{equation*}

Since $\int n_1 \,d\mu <+\infty$, using Theorem 1.1 of \cite{zweimuller}  we have that there is an ergodic $g$-invariant measure $\mu_g$ absolutely continuous with respect to $\mu$. Moreover, if $G_k$ is the subset of points $\ud x\in G$ such that $n_1(\ud x)=k$ then, we may characterize this measure defining 

\begin{equation}\label{eq.induzvolta}
\mu(E) = \sum_{n=0}^{\infty} \sum_{k>n} \mu_g(\sigma^{-n}(E)\cap G_k),
\end{equation}
for every measurable set $E \subset \Sigma_1$. Thus, by Birkhoff Ergodic Theorem applied to the system $(g,\mu_g)$, we have that for $\mu_g$ almost everywhere $\ud x$
\begin{equation}%\label{eq.nki}
\lim\limits_{k\rightarrow \infty} \frac{1}{k}\sum\limits_{j=0}^{k-1} n_1(g^j(\ud x)) = \int n_1 \,d\mu_g.
\end{equation}

Observe that $n_k(\ud x) = \sum\limits_{j=0}^{k-1} n_1(g^j(\ud x))$. Consequently,  we have that

\begin{equation}\label{eq.nki}
\lim\limits_{k\rightarrow \infty} \frac{n_k(\ud x)}{k}=\lim\limits_{k\rightarrow \infty} \frac{1}{k}\sum\limits_{j=0}^{k-1} n_1(g^j(\ud x)) =\int n_1 \,d\mu_g,
\end{equation} and this finish the proof.

\end{proof}

\section{Conformal Measures and Weak Gibbs Measures}
 We observe that eigen-measures for the Ruelle-Perrón-Frobenius operator are  natural candidates for sequential Gibbs measures.  To recall,  denote by $\mathcal{C}(\Sigma)$ the set of real-valued continuous functions on $\Sigma$. The Ruelle-Perron-Frobenius operator $\mathcal{L}_{\psi}:\mathcal{C}(\Sigma)\rightarrow \mathcal{C}(\Sigma)$ associated to $\psi \in \mathcal{C}(\Sigma)$ is defined by
 
 \[\mathcal{L}_{\psi}\phi(\ud x):=\sum_{\ud y \in \sigma^{-1}(\ud x)}e^{\psi(\ud y)}\phi(\ud y).\] 
 Observe, that for $n\in \mathbb{N}$
 
 \[\mathcal{L}_{\psi}^n\phi(\ud x)=\sum_{\ud y \in \sigma^{-n}(\ud x)}e^{\psi^n(\ud y)}\phi(\ud y)\]
 
 This operator is \emph{positive}, i.e.,  preserves the cone of positive functions $C(\Sigma)^+$.  Therefore, we may restrict the dual operator $\mathcal{L}_{\psi}^{*}$  to the dual cone $(C(\Sigma)^+)^{*}$. If we identify the cone $(C(\Sigma)^+)^{*}$ with the space of positive finite measures $\mathcal{M}(\Sigma)$ by Riesz Theorem, the operator  $\mathcal{L}_{\psi}^{*}$ is defined by:
 
 \begin{eqnarray*}
 	\mathcal{L}_{\psi}^{*}:\mathcal{M}(\Sigma)&\rightarrow&  \mathcal{M}(\Sigma)\\
 	\mu&\mapsto& \mathcal{L}_{\psi}^{*}(\mu):\mathcal{C}(\Sigma)\rightarrow \mathbb{R}\\
 	&&\;\;\;\;\;\;\;\;\;\;\;\;\;\;\;\;\;\;\;\;\phi\mapsto \mathcal{L}_{\psi}^{*}(\mu)
 	(\phi)=\int_{\Sigma}\mathcal{L}_{\psi}(\phi)d\mu
 \end{eqnarray*}
 %
 %Let $\Psi:\mathcal{M}_1(\Sigma)\rightarrow \mathcal{M}_1(\Sigma)$ be the operator defined on the space $\mathcal{M}_1(\Sigma)$ of probabilities by
 %
 %\[\Psi(\mu)=\left(\int_\Sigma \mathcal{L}_{\psi}(1)d\mu\right)^{-1}\cdot\,\mathcal{L}_{\psi}^{*}(\mu)\]
 
 If $r(T)$ denotes the spectral radius of $T$, we have that $r(\mathcal{L}_{\psi})=r(\mathcal{L}_{\psi}^{*})$.  Since $C(\Sigma)^+$ is a normal cone with non-empty interior, the  spectral theory of positive operators on cones (see \cite{edmunds} or \cite{baladi-book}, for instance) give us that
 that $r(\mathcal{L}_{\psi})$ is an eigenvalue of $\mathcal{L}_{\psi}^{*}$ with some eigenvector $\mu \in \mathcal{M}(\Sigma)$, i.e,  
 
 \[ \mathcal{L}_{\psi}^{*}\mu=r(\mathcal{L}_{\psi})\mu.\]  We say these measures  are \emph{conformal measures} with respect to $\psi$. A main feature of a conformal measure is the relation between the variation of $\psi$ and distortion properties:
 
 \begin{proposition}
 	Let $\mu \in \mathcal{M}(\Sigma)$ be a conformal measure for $\psi \in \mathcal{C}(\Sigma)$. Then for every
 	$n \in \mathbb{N}$  and $\ud y\in \sigma^{-j}([x_j...x_{n-1}])$, with $0\leq j <n$, we have that 
 	
 	\begin{equation}\label{eq.gibbsvar} e^{-var_{n-j}(\psi^{{n-j}},\,\sigma^j(\ud x))}\leq \frac{\mu([x_j...x_{n-1}])}{e^{\psi^{n-j}(\sigma^j(\ud y))-(n-j)P}}
 	\leq e^{var_{n-j}(\psi^{{n-j}},\,\sigma^j(\ud x))}, 
 	\end{equation} where $\log r(\mathcal{L}_\psi)= P$.
 	
 \end{proposition}
 \begin{proof}
 	For all $[x_j...x_{n-1}]$ and $\ud y\in \sigma^{-j}([x_j...x_{n-1}])$, with $0\leq j<n$, we have
 	\begin{eqnarray*}
 		\mu([x_j...x_{n-1}])&=&\int 1_{[x_j...x_{n-1}]}d\mu
 		=\lambda^{-(n-j)}\int \mathcal{L}_{\psi}^{n-j}1_{[x_j...x_{n-1}]}d\mu\\
 		&=&\lambda^{-(n-j)}\int \sum_{\sigma^j(\ud y) \in \sigma^{-(n-j)}(\ud z)}1_{[x_j...x_{n-1}]}(\sigma^j(\ud y)) e^{\psi^{n-j}(\sigma^j(\ud y))}d\mu(\ud z)\\
 		&\leq& \lambda^{-(n-j)}e^{var_{n-j}(\psi^{n-j},\,\sigma^j(\ud x))}e^{\psi^{n-j}(\sigma^j(\ud x))}
 	\end{eqnarray*}
 	
 	similarly
 	\[\lambda^{-(n-j)}e^{-var_{n-j}(\psi^{n-j},\,\sigma^j(\ud x))}e^{\psi^{n-j}\sigma^j(\ud x))}
 	\leq \mu([x_j...x_{n-1}])\]
 	
 	Then,
 	
 	\[e^{-var_{n-j}(\psi^{{n-j}},\,\sigma^j(\ud x))}\leq \frac{\mu([x_j...x_{n-1}])}{e^{\psi^{n-j}(\sigma^j(\ud x))-(n-j)P}}
 	\leq e^{var_{n-j}(\psi^{{n-j}},\,\sigma^j(\ud x))}.
 	\]
 \end{proof}

 To analyze sufficient conditions such that a conformal measure is sequential Gibbs measure, we  define  the sequence of functions  $\xi_n(\ud x)$ at  $\ud x = x_0 x_1...$ by
 
 $$\xi_n(\ud x)=\sup_{\ud y \in [x_{0}...x_{n-1}]}\left\{\sum_{i=0}
 ^{n-1}|\psi(\sigma^{j}(\ud x)-\psi(\sigma^{j}(\ud y)|\right\}.$$ 
 
 Now, we introduce an useful proposition to allow us to check that a conformal measure is sequential Gibbs. This proposition is used in Example \ref{ex.hofbauer}.
 \begin{proposition}\label{prop corformexsgibbs} 
 	Given a conformal measure $\mu$ such that  $\liminf_{n\rightarrow \infty} \xi_n(\ud x)\leq C$ at $\mu$-a.e $\ud x \in \Sigma$, for some constant $C>0$,  then  $\mu$ is a sequential Gibbs measure.
 \end{proposition}
 \begin{proof}

 	Observe that   $var_n(\psi^n,\ud x)\leq \xi_n(\ud x)$ and 
 	$\xi_{n-j}(\sigma^j(\ud x))\leq \xi_n(\ud x)$, for each $0\leq j\leq n$. Then,
 	
 	\[var_{n-j}(\psi^{n-j},\,\sigma^{j}(\ud x))\leq \xi_{n-j}(\sigma^j(\ud x))\leq \xi_n(\ud x)\]
 	
 	By hypothesis, for almost every $\ud x\in \Sigma$ there is a sequence $n_i(\ud x)$ such that $\xi_{n_i}(\ud x)\leq C$. Thus,  by  Equation~\eqref{eq.gibbsvar},  we have for any $n_i(\ud x)$
 	\begin{equation}\label{eq.varx}
 	e^{-C} \leq e^{-var_{n_i-j}(\psi^{{n-j}},\,\sigma^j(\ud x))}\leq \frac{\mu([x_j...x_{n_i-1}])}{e^{\psi^{n_i-j}(\sigma^j(\ud y)-(n_i-j)P}}
 	\leq e^{var_{n_i-j}(\psi^{{n_i-j}},\,\sigma^j(\ud x))} \leq e^C. 
 	\end{equation}
 	And this finish the proof. 
%\footnote{O R2 fala sobre essa proposição. Não entendi muito bem...mas acho q pede outros exemplos} 	
 \end{proof}
 
We discuss some conditions on the sequence $(\xi_n)_{n\geq 1}$ that give us more information about conformal sequential Gibbs measures. The results here will not be used elsewhere in this paper and are included to help to clarify the relation between our notion of sequential Gibbs measures and weak forms of Gibbs measures studied before.  We begin recalling the notion of weak Gibbs measure for continuous potentials, studied by Yuri in \cite{yuri_1999}:

\begin{definition} \label{def.yuri}
	A measure $\mu$ is a \textit{weak Gibbs} measure for the potential  $\psi:\Sigma\rightarrow \bb R$ if there is a constant $P$ and  a sequence of positive numbers  $K_n$ satisfying
\begin{eqnarray}\label{eq seq weak gibbs}
\lim_{n\rightarrow \infty}\frac{\log K_n}{n}=0,
\end{eqnarray}
such that  for each $n\in \bb N$, $\ud x = x_0x_1...$ and $\ud y\in [x_0...x_{n-1}]$ we have
	
	\begin{equation}\label{eq weak gibbs yuri}
	\frac{1}{K_n}\leq \frac{\mu([x_0...x_{n-1}])}{e^{\psi^n(\ud y)-nP}}\leq K_n
	\end{equation}
	
Following \cite{yuri_1999}, we say that	$\psi$ is of \textit{weak bounded variation}(WBV) if there exists a sequence of positive numbers  $K_n$ satisfying (\ref{eq seq weak gibbs})\, such that
	\begin{eqnarray*}
	\sup_{[x_0...x_{n-1}]\in \mathcal{C}_n}\;\sup_{\ud y, \ud w \in [x_0...x_{n-1}]}\frac{e^{\psi^n(\ud y)}}{e^{\psi^n(\ud w)}}\leq K_n
	\end{eqnarray*}
where $\mathcal{C}_n$ the collection of all cylinders $C=[x_0\dots x_{n-1}]$ of length $n$.
	\end{definition}
	
	In \cite{yuri_1999}, the author discussed some examples of nonuniformly hyperbolic maps with potentials with an unique equilibrium measure that fails to be a Gibbs measure, but has the weak Gibbs property. Now, we establish, as in \cite{yuri_1999}, the relation between sequential and weak Gibbs conformal measures, using the sequence $\xi_n$:
	%\footnote{R1. 3.3 implica WBV...coloquei a definição e incluí na proposição a prova}
	\begin{proposition}
	If \, $\lim_{n\rightarrow \infty} (1/n)\|\xi_n\|_\infty=0$, then any conformal measure is a weak Gibbs measure and $\psi$ is WBV.
\end{proposition}
\begin{proof}
	Following the steps of the proof of the Proposition \ref{prop corformexsgibbs} and observing the Equation~\eqref{eq.varx}, we have that
	\begin{eqnarray*}
		e^{-\xi_n(\ud x)} &\leq&e^{-var_{n}(\psi^{{n}},\,\ud y)}\leq \frac{\nu([x_0,...,x_{n-1}])}{e^{\psi^{n}(\ud y)-nP}}\\
		&\leq& e^{var_{n}(\psi^{{n}},\,\ud y)}\leq e^{\xi_n(\ud x)}
	\end{eqnarray*}
	for every $\ud y\in [x_0...x_{n-1}]$. Then, put $K_n:=e^{\|\xi_n\|_\infty}$ we have for $\ud y,\ud w \in [x_0...x_{n-1}]$ 
	
	\[\frac{e^{\psi^n(\ud y)}}{e^{\psi^n(\ud w)}}\leq K_n^2\]
To end the proof of the proposition just take the supremum.
	
\end{proof}

Now, we discuss the notion of \emph{non-lacunary Gibbs measure}, studied in \cite{OV}. We say that a sequence of natural numbers  $a_1<a_2<...$ is \emph{non-lacunary}, if $\lim_{i\rightarrow \infty} a_{i+1}/a_i =0$. 

A non-lacunary Gibbs measure is a sequential Gibbs measure such that the sequence $n_i(\ud x)$ is \emph{non-lacunary} at almost every point $\ud x \in \Sigma$.  The proof of next lemma follows, mutatis mutandis, from the proof of Proposition~$3.8$ of \cite{OV}. The key property used here is the $n_1(\sigma^{n_i}(\ud x))=n_{i+1}(\ud x)-n_i(\ud x)$. We will include it here for the sake of completeness. 

For $m\geq 1$, let $G_m=\{\ud x \in \Sigma: n_1(\ud x)=m\}$.
\begin{lemma} \label{l.nonlacunary}
	Let $\mu$ be a invariant sequential Gibbs measure such that the function $n_1$ is integrable. Then, for almost every $\ud x \in \Sigma$, the sequence $n_i(\ud x)$ is non-lacunary. 
	\end{lemma}

\begin{proof}
Let $D\subset \Sigma$ be the set of points which the sequence $n_j(\cdot)$ fails to be non-lacunary. For each $r>0$, define $L_r(n)=\{\ud x \in \Sigma:n_1(\ud x)\geq r n\}$. If $\ud x \in D$ then there exists a rational number $r>0$, and there are infinitely many values of $i$ such that $n_{i+1}(\ud x)-n_i(\ud x)\geq r n_{i}(\ud x)$. Then,

$$
n_1(\sigma^{n_i}(\ud x))=n_{i+1}(\ud x)-n_i(\ud x)\geq r n_i(\ud x).
$$
So, there are arbitrarily large values of $n$ such that $\ud x \in \sigma^{-n}(L_r(n))$. In the words, $D$ is contained in the set
$$
L=\bigcup_{r\in \bb Q\,\cap (0,+\infty)}\bigcap_{k= 0}^{\infty}\bigcup_{n\geq k}\sigma^{-n}(L_r(n)).
$$
By invariance of $\mu$, we have for all $n$ $\mu(\sigma^{-n}(L_r(n)))=\mu(L_r(n))$. Then
$$
\sum_{n=1}^{\infty}\mu(L_r(n))=\sum_{n=1}^{\infty}\sum_{n_1\geq rn}\mu(G_{n_1})=\sum_{m=1}^{\infty}\sum_{n=1}^{[m/r]}\mu(G_m)\leq \sum_{m=1}^{\infty}(m/r)\mu(G_m).
$$
Using that $n_1(\cdot)$ is integrable, we have
$$
\sum_{n=1}^{\infty}\mu(L_r(n))\leq \frac{1}{r}\int n_1(\ud x) d\mu(\ud x)<\infty.
$$
By the Borel-Cantelli lemma, this implies that $L$ has measure zero. It follows that $\mu(D)=\mu(L)=0$, as claimed.

\end{proof}

In view of Lemma~\ref{l.nonlacunary} and following the proof of Proposition~$3.17$ of \cite{OV}, we are able to show that: 	
	
	\begin{proposition}\label{p.gibbsnaolacunar}
	If $\mu$ is sequential Gibbs measure and the function $n_1$ is integrable, then exist a sequence of positive functions
	$K_n>1$  such that $\mu$-a.e. $\ud x$ and for all  $n\in \mathbb{N}$, we have
	\begin{eqnarray}\label{weak gibbs} 
	K_n^{-1}(\ud x)\leq \frac{\mu([x_0...x_{n-1}])}{e^{\phi^n(\ud y)-nP}}\leq K_n(\ud x)
	\end{eqnarray}
	 and  $\limsup_{n\rightarrow \infty}\frac{\log K_n(\ud x)}{n}=0$.
\end{proposition}

\begin{remark} We observe that given a conformal sequential Gibbs measure, it is always possible to choose $K_n(x)$ as in Proposition~\ref{p.gibbsnaolacunar}  in such way that $(1/n)\log K_n(\ud x)$ converges almost everywhere. 

Indeed, since $\xi_n$ is a  subadditive sequence of non-negative functions, we use the Ergodic Subadditive Theorem of  Kingman (see \cite{OV-book}, Theorem 3.3.3) to have that $(1/n) \xi_n$ converge almost everywhere. Then, take $K_n (x) :=e^{\xi_n(\ud x)}$ and observe that the Equation~\eqref{weak gibbs} is  satisfied.
\end{remark}

\section{Examples}\label{examples}

 In this section we discuss some examples of sequential Gibbs measures. The first one  was  introduced  in \cite{Hofbauer}, where Hofbauer gave an interesting example of a family of continuous potentials with phase transitions  and equilibrium states that are not a Gibbs measures. We reproduce this example and show that, despite the fact that they not satisfy the Gibbs property,   these measure are sequential Gibbs measure with integrable first Gibbs time function.

\begin{example}\label{ex.hofbauer}
	
	For simplicity, let $\Sigma_2^{+}=\{1,0\}^\mathbb{N}$ be an one-sided shift space with two symbols.  Consider the partition of $\Sigma_2^+$ by sets $(M_k)_{k\geq 0}$ plus the point $\ud 1= 1111..$ where $M_k$ is defined by
	$M_0=[0]$ and for $k=1,2,...$ 
	$$M_k=[\underbrace{11..1}_{k \text{ times}}0]=\{\ud x\in \Sigma_2^{+}:x_i=1 \,\,\mbox{for}\,\,
	0\leq i< k-1,\, x_k= 0\}.
	$$
	
	Let $(a_k)$ be a sequence of real numbers with $\lim a_k=0$. Set $s_k=a_0+...+a_k$. Define a continuous potential  $g\in C(\Sigma_2^{+})$ by
	
	\[g(\ud x)=a_k\,\,\mbox{for}\,\, \ud x\in M_k\,\, \mbox{and}\,\, g(11...)=0.\]
\end{example}

As was pointed out at Section~\ref{s.results},  there exists some conformal measure $\nu$ with respect to $g$. By results in (\cite{Hofbauer}, page 230), $g$ admits a Gibbs measure if, and only if, $\sum_{k\geq 0} a_k$ is convergent. Assume that $g$ has no Gibbs measures, i.e.,  $\sum_{k\geq 0} a_k$ diverges. 

If $\sum_{k\geq 0} e^{s_k}>1$, by (\cite{Hofbauer}, page 226) we have  that  there exists some positive continuous function $h$ such that  $\mu=h\nu$ is the unique equilibrium state of $g$. We prove that in this case, despite the fact that it do not satisfy the Gibbs property,   $\mu$ is a sequential Gibbs measure. 

Indeed, since $\mu=h \nu$ and $h$ is bounded from above and below,  it follows from the fact  that $\nu$ is a sequential Gibbs measure. As we observed before in Proposition~\ref{prop corformexsgibbs},  it is enough to show that $\liminf \xi_n(\ud x)$ is bounded at $\nu$ almost everywhere. In fact, we prove that $\liminf \xi_n(\ud x)=0$ at $\nu$ almost every point. 

Firstly, observe that from the definition, $\xi_{k+1}(\ud x)=0$, if $ \sigma^{k}(\ud x) \in M_0$.  On the other hand, since $\mu \neq \delta_{11...}$, we have that  for almost every point $\ud x$ there exists a sequence $k_1(\ud x)<k_2(\ud x)<\dots$ such that  $\sigma^{k_i}(\ud x) \in M_0$. In particular, $\xi_{k_i+1}(\ud x)=0$. 

From this,   we have that the first Gibbs time is integrable. Indeed,  using Proposition~\ref{prop corformexsgibbs}, we have that the first Gibbs time function $n_1$ of $\mu$ is smaller than the first return time to $M_0$. Then, by Kac's Lemma, we have that $n_1$ is integrable with respect to $\mu$.

In the next example we discuss  the non-lacunary Gibbs measures studied in \cite{OV,VV,ramosviana}. These measures are equilibrium states for H\"older \emph{hyperbolic} potentials  of some $C^1$ local diffeomorphisms on compact Riemannian manifolds and they have only positive Lyapunov exponents.

\begin{example}\label{ex.localdiffeos}  Let $f:M\rightarrow M$ be an $C^1$ local diffeomorphism  of a compact connected manifold $M$ such that there exists  sets $R_1 , . . . ,R_q$ of $M$ that are domains of injectivity of $f$   such that $\overline{R_i}\cap \overline{R_j} = \emptyset$, for $i \neq j$,   and  $f(R_i)=M$, for $ 1\leq i \leq q$. Consider $\mathcal{R}=R_1 \cup \dots \cup R_q$ and the invariant set $\Lambda = \cap_{n\geq 1}  f^{-n}(\mathcal{R})$. We may define a semiconjugacy 

$$
\pi:\Lambda \rightarrow \Sigma_q^+,
$$
 between $f|_\Lambda$ and $\sigma:\Sigma_q^+\rightarrow \Sigma_q^+$, where $\Sigma_q^+=\{1,\dots,q\}^\mathbb{N}$,  just considering $\pi(x)$ as the itinerary of $x$ with respect to the partition $\mathcal{P}=\{P_1,\dots,P_q\}$ of $\Lambda$, defined by $P_i=R_i \cap \Lambda$. Denote by $P(x)$ the element of $\mathcal{P}$ that contains  $x$ and assume that there exists $\sigma_1, \sigma_2>1$ such that:

\begin{itemize}

\item $\|Df(x)^{-1}\|<\sigma_2$, for every $x \in M$ and

\item $\|Df(x)^{-1}\|<\sigma_1^{-1}<1$, for $x$ in the complement of an open set containing $R_1$. 

\end{itemize}

In \cite{OV} and \cite{VV}, the authors proved that if $\phi$ is a H\"older continuous potential such that $\max \phi - \min \phi$ is small enough, then there exists an unique equilibrium state $\eta$ for $\phi$ and this measure has only positive Lyapunov exponents and it is  a non-lacunary Gibbs measure, in the sense that if we define  
$$
P^n(x)=P(x)\cap f^{-1}(P(f(x)))\cap \dots \cap f^{-(n-1)}(P(f^{n-1}(x))),
$$ then there exist a constant $K$, such that  for $\eta$  almost every $x \in \Lambda$   there exists a sequence $n_i(x)$ such that $\lim_{i\rightarrow \infty} n_{i+1}(x)/n_i(x)=1$ and   

\[
K^{-1}\leq \frac{\eta(P^{n_k}(x))}{e^{\phi^{n_i}(x)-n_i P(\phi) }}\leq K.
\] 

If we consider the push-forward measure $\mu=\pi_* \eta$ on $\Sigma_q^+$, then the map $\pi$ is invertible in a set of $\mu$-full measure and the measure $\mu$ is a sequential Gibbs measure with respect to the potential $\psi=\phi\circ \pi^{-1}$.
\end{example}

\section{Proof of Theorem~\ref{fs.teo proj}}
%\footnote{R1 pede uma redução da prova}
In this subsection we construct the potential $\psi_{2}$ as in Theorem~\ref{fs.teo proj},  obtained as the limit of a converging sequence of functions. 
Given $z\in E\subset  \Sigma_2$, as in the Definition \ref{fs hyp}, denote by $(n_{k}(\ud z))_{i\geq 1}$  the sequence of  Gibbs times of any $\underline{x} \in \Pi^{-1}(\underline{z})\cap \Sigma_1$. By Hypothesis \ref{fs hyp}, $n_k(\ud z)$ is well defined for a set of full $\nu$ measure.

\begin{definition}\label{fs.def3}
Given $k\in \mathbb{N}$ and $\underline{w}\in \Sigma_{1}$, define $u_{\underline{w},k}:E\subset \Sigma_{2}\rightarrow \mathbb{R}$ by

$$
u_{\underline{w},k}(\underline{z})= \frac{\sum_{\underline{x}=x_{0}...x_{n_{k}}} e^{\psi_{1}^{n_{k}+1}(\underline{x}\underline{w})}}{\sum_{\underline{x}'=
x_{1}...x_{n_{k}}} e^{\psi_{1}^{n_{k}}(\underline{x}'\underline{w})}},
$$
  where $\sum_{\underline{x}=x_{0}...x_{n_{k}}}$ represents the sum over finite words $\underline{x}=x_0x_1\dots x_{n_k}$ such that $\pi(x_{i})=z_{i}$, for $i=0,...,n_{k}$ and $\underline{x}\underline{w}=x_{0}...x_{n}w_{0}w_{1}...$.
\end{definition}

We show that 
%\footnote{R1. comentou algo nessa página...mas não entendi}
\begin{proposition}\label{fs.prep.1.1}
The limit $u(\underline{z}):=\lim_{k\rightarrow \infty}u_{\underline{w},k}(\underline{z})$ is well defined and independent of $\underline{w}$.
\end{proposition}

The proof of Proposition \ref{fs.prep.1.1} is the central point of this article. We postpone this proof to the next section. However,  assume it to be true for a moment and let us prove that $\nu$ is a sequential Gibbs measure for $\psi_2=\log u$. First, we prove that:

\begin{lemma}\label{fs.lema 2.4}
There is a constant $C>0$ depending only $\psi_{1}$, such that for every $\underline{w},\ud{w}'$ and a sequence of Gibbs times $(n_{i}(\underline{x}))_{i\geq 1}$  and $0\leq l\leq n_{i}$, we have
 
\[\frac{e^{\psi_{1}^{n_{i}-l+1}(\sigma^{l}(\underline{x}\underline{w}))}}{e^{\psi_{1}^{n_{i}-l+1}(\sigma^{l}(\underline{x}\underline{w}'))}}\leq C\]
\end{lemma}
\begin{proof}

Note that the definition of sequential Gibbs measures we have for every choice $\underline{w}$, $\underline{w}'$ and  $\underline{x}=x_{0}...x_{n_{i}}$ (with $0\leq l \leq n_{i}$),

\[K^{-1}e^{\psi_{1}^{n_{i}-l+1}(\sigma^{l}(\underline{x}\underline{w}))}\leq \mu[x_{l}...x_{n_{i}}]\leq Ke^{\psi_{1}^{n_{i}-l+1}(\sigma^{l}(\underline{x}\underline{w}'))}\]

Thus,

\[\frac{e^{\psi_{1}^{n_{i}-l+1}(\sigma^{l}(\underline{x}\underline{w}))}}{e^{\psi_{1}^{n_{i}-l+1}(\sigma^{l}(\underline{x}\underline{w}'))}}\leq K^2=C\]

\end{proof}

\begin{corollary}\label{fs.cor 2.4}
 With the same hypothesis of the previous lemma, we have that there is a constant 
 $C>0$, depending only $\psi_{1}$, such that for every $n_{i}$ and $0\leq l\leq n_{i}$, we have

 \[\frac{\sum_{\underline{x}=x_{l}...x_{n_{i}}}e^{\psi_{1}^{n_{i}-l+1}(\underline{x}\underline{w})}}
 {\sum_{\underline{x}=x_{l}...x_{n_{i}}}e^{\psi_{1}^{n_{i}-l+1}(\underline{x}\underline{w}')}}\leq C\]
 \end{corollary}

We can now define the potential for $\nu$.

\begin{definition}
 We define the potential $\psi_{2}:\Sigma_{2}\rightarrow \mathbb{R}$ by $\psi_{2}(\underline{z}):=\log u(\underline{z})$.
\end{definition}
The main problem is precisely to show that the potential $\psi_{2}$ is well defined. We follow the lines of \cite{Kempton-Pollicot} and also prove the  assertions of Theorem \ref{fs.teo proj}. Suppose, for a moment, that the Proposition~\ref{fs.prep.1.1} is true. We have the following lemma.

\begin{lemma}
 The measure $\nu=\Pi_*\mu$ is sequential Gibbs measure for the potential $\psi_{2}(\underline{z})=\log u(\underline{z})$.
\end{lemma}
\begin{proof}
 Let us fix $n\geq 1$. We can write
 
 \begin{align*}  
\psi_{2}^{n+1}(\underline{z})&=\sum_{i=0}^{n}\log u(\sigma^{i}
 (\underline{z}))=\lim_{k\rightarrow \infty}\log u_{\underline{w},k}(\underline{z})+...+\lim_{k\rightarrow \infty}
 \log u_{\underline{w},k}(\sigma^{n}(\underline{z}))
 \end{align*}
 
Since $n_k(\ud z)-l$ is a Gibbs time of $\sigma^l(z)$ ,  given $1\leq l\leq n_k$, we may choose
a  subsequence $(i_{k}^{l})_{k\geq 1}$, such that
$n_{i_{k}^{l}}(\sigma^{l}(\underline{z}))=n_{k}(\underline{z})-l.$  Consequently, given $n$

 \begin{align*}
\psi_{2}^{n+1}(\underline{z})&= \lim_{k\rightarrow \infty}\log \left(u_{\underline{w},i_{k}^{0}(\underline{z})}(\underline{z})\cdot... \cdot 
u_{\underline{w},i_{k}^{n}(\sigma^{n}(\underline{z}))}(\sigma^{n}(\underline{z})\right)=\\
&=\lim_{k\rightarrow \infty}\log\left(\frac{\sum_{\underline{x}=x_{0}...x_{n_{k}}}e^{\psi_{1}^{n_{k}+1}(\underline{x}\underline{w})}}
{\sum_{\underline{x}'=x_{1}...x_{n_{k}}}e^{\psi_{1}^{n_{k}}(\underline{x}'\underline{w})}}\cdot 
\frac{\sum_{\underline{x}=x_{1}...x_{n_{k}}}e^{\psi_{1}^{n_{k}}(\underline{x}\underline{w})}}{\sum_{\underline{x}'=x_{2}...x_{n_{k}}}
e^{\psi_{1}^{n_{k}}(\underline{x}'\underline{w})}}\cdot...\right.\\
&\cdot \left.\frac{\sum_{\underline{x}=x_{n}...x_{n_{k}}}e^{\psi_{1}^{n_{k}-n+1}(\underline{x}\underline{w})}}{\sum_{\underline{x}'=
x_{n+1}...x_{n_{k}}}e^{\psi_{1}^{n_{k}-n}(\underline{x}'\underline{w})}}\right)=\lim_{k\rightarrow \infty}\log \left(\frac{\sum_{\underline{x}=
x_{0}...x_{n_{k}}}e^{\psi_{1}^{n_{k}+1}(\underline{x}\underline{w})}}{\sum_{\overline{x}=x_{n+1}...x_{n_{k}}}e^{\psi_{1}^{n_{k}-n}
(\overline{x}\underline{w})}}\right)
\end{align*}

Note that, in the same way for $1\leq l\leq n$, we have

\begin{eqnarray}\label{fs.des 2}
\psi_{2}^{n-l+1}(\sigma^{l}(\underline{z}))=\lim_{k\rightarrow \infty}\log \left(\frac{\sum_{\underline{x}=
x_{l}...x_{n_{k}}}e^{\psi_{1}^{n_{k}-l+1}(\underline{x}\underline{w})}}{\sum_{\overline{x}=x_{n+1}...x_{n_{k}}}e^{\psi_{1}^{n_{k}-n}
(\overline{x}\underline{w})}}\right)
\end{eqnarray}

Moreover, for $n_{k} > n_{i}$ and for $0\leq l\leq n_{i}$ we can write

\begin{eqnarray*}
 \sum_{\underline{x}=x_{l}...x_{n_{k}}}e^{\psi_{1}^{n_{k}-l+1}(\underline{x}\underline{w})}
=\sum_{\underline{x}=x_{l}...x_{n_{i}}}\sum_{\overline{x}=x_{n_{i}+1}...x_{n_{k}}}
e^{\psi_{1}^{n_{i}-l+1}(\underline{x}\overline{x}\underline{w})}e^{\psi_{1}^{n_{k}-n_{i}}
(\overline{x}\underline{w})}
\end{eqnarray*}

By Corollary \ref{fs.cor 2.4}, for $0\leq l \leq n_i$ we have

\[\sum_{\underline{x}=x_{l}...x_{n_{k}}}e^{\psi_{1}^{n_{k}-l+1}(\underline{xw})}\leq C\sum_{\underline{x}'=x_{l}...x_{n_{i}}}
e^{\psi_{1}^{n_{i}-l+1}(\underline{x}'\underline{w})}\sum_{\overline{x}=x_{n_{i}+1}...x_{n_{k}}}e^{\psi_{1}^{n_{k}-n_{i}}
(\overline{x}\underline{w})},\]
and also follows from Corollary \ref{fs.cor 2.4}

\begin{eqnarray}
 C^{-1}\cdot\sum_{\underline{x}'=x_{l}...x_{n_{i}}} e^{\psi_{1}^{n_{i}-l+1}(\underline{x}'\underline{w})}\leq
\frac{\sum_{\underline{x}=x_{l}...x_{n_{k}}}e^{\psi_{1}^{n_{k}-l+1}(\underline{x}\underline{w})}}{\sum_{\overline{x}=
x_{n_{i}+1}...x_{n_{k}}}e^{\psi_{1}^{n_{k}-n_{i}}(\overline{x}\underline{w})}}
 \leq
C\cdot \sum_{\underline{x}'=x_{l}...x_{n_{i}}}e^{\psi_{1}^{n_{i}-l+1}(\underline{x}'\underline{w})}\nonumber
\end{eqnarray}

Taking $k\rightarrow\infty$, we have

\[ C^{-1}\cdot\sum_{\underline{x}'=x_{l}...x_{n_{i}}} e^{\psi_{1}^{n_{i}-l+1}(\underline{x}'\underline{w})}\leq
e^{\psi_{2}^{n_i-l+1}(\sigma^{l}(\underline{z}))}
 \leq C\cdot \sum_{\underline{x}'=x_{l}...x_{n_{i}}}e^{\psi_{1}^{n_{i}-l+1}(\underline{x}'\underline{w})}
\]

And so

\[e^{\psi_{2}^{n_i-l+1}(\sigma^{l}(\underline{z}))}\approx \sum_{\underline{x}'=x_{l}...x_{n_{i}}}e^{\psi_{1}^{n_{i}-l+1}(\underline{x}'\underline{w})}\]

Since $\mu$ is an sequential Gibbs measure  for $\psi_{1}$ %there are constant $K$, such that 
then
for each $\underline{x} \in \Pi^{-1}(\underline{z})\cap \Sigma_1$ and an sequence $(n_{i}(\underline{x}))_{i\geq 1}$,

%\[K^{-1}e^{\psi_1^{n_{i}+1-l}(\sigma^{l}(\underline{x}))}\leq \mu_{1}[x_{l}...x_{n_{i}}]\leq K e^{\psi_1^{n_{i}+1-l}(\sigma^{l}(\underline{x}))}.\]

$$
e^{\psi_1^{n_{i}+1-l}(\sigma^{l}(\underline{x}))}\approx \mu_{1}[x_{l}...x_{n_{i}}]
$$

for $0\leq l\leq n_{i}$. Adding over all words $\underline{x}$ that are projected in $\underline{z}$, we have

%\begin{eqnarray*}
% K^{-1}\sum_{\underline{x}'=x_{l}...x_{n_{i}}} e^{\psi_{1}^{n_{i}-l+1}(\underline{x}'\underline{w})}&\leq&\sum_{x_{l}...x_{n_{i}}}
% \mu[x_{l}...x_{n_{i}}]\\
% &\leq& K\sum_{\underline{x}'=x_{l}...x_{n_{i}}}e^{\psi_{1}^{n_{i}-l+1}(\underline{x}'\underline{w})}
%\end{eqnarray*}

%So, 

%\[\sum_{x_{l}...x_{n_{i}}}\mu[x_{l}...x_{n_{i}}]\approx\sum_{\underline{x}'=x_{l}...x_{n_{i}}} e^{\psi_{1}^{n_{i}-l+1}(\underline{x}'\underline{w})}\]

%Therefore,

\[\nu[z_{l}...z_{n_{i}}]=\sum_{x_{l}...x_{n_{i}}}\mu[x_{l}...x_{n_{i}}]\approx e^{\psi_{2}^{n_i-l+1}(\sigma^{l}(\underline{z}))}\]
for each $0\leq l\leq n_{i}$, proving that $\nu$ is sequential Gibbs measure for  $\sigma:\Sigma_{2}\rightarrow \Sigma_{2}$ and  $\psi_2$.
\end{proof}

In this section, we prove that $\psi_{2}$ is well defined. We give definitions that  help us in this purpose.

\begin{definition}
 Let be $k\in \mathbb{N}$ and $\underline{z}\in E\subset \Sigma_{2}$. We define the closed interval
 \[\Lambda_{k}(\underline{z}):=\left[\min_{\underline{w}}u_{\underline{w},k}(\underline{z}),\max_{\underline{w'},}
 u_{\underline{w'},k}(\underline{z})\right].\]

 Given $k\in \mathbb{N}$ e $\underline{z}\in E\subset\Sigma_{2}$. We define
\[\lambda_{k}(\underline{z}):=\sup\left\{ \frac{u_{\underline{w},k}(\underline{z})}{u_{\underline{w}',k}(\underline{z})}
:\underline{w},\underline{w'}\in \Sigma_{1}\right\}.\]
\end{definition}

We say that a sequence of intervals $I_{n}$ is monotonically nested if we have
 \[I_{0}\supseteq I_{1}\supseteq ... \supseteq I_{n}\supseteq...\]

In the next lemma, we show that the sequence $(\Lambda_{k}(\underline{z}))_{k\geq 1}$ is monotonically nested. Then, the existence of
$\psi_{2}$ at $\nu-$a.e. $\underline{z}\in \Sigma_{2}$ corresponds to the convergence to  1 of the sequence 
$\lambda_{k}(\underline{z})$.

\begin{lemma}\label{fs. nested}
The sequence of intervals $(\Lambda_{k}(\underline{z}))_{k\geq 1}$ is monotonically \textit{nested}. 
\end{lemma}
\begin{proof}
Given $z\in E$, observe that 

\begin{eqnarray*}
u_{\underline{w},k+1}(\underline{z})&=&\frac{\sum_{\underline{x}=x_{0}...x_{n_{k+1}}}e^{\psi_{1}^{n_{k+1}+1}
(\underline{x}\underline{w})}}{\sum_{\underline{x'}=x_{1}...x_{n_{k+1}}}
e^{\psi_{1}^{n_{k+1}}(\underline{x}'\underline{w})}}\\
&=&\frac{\sum_{\underline{x}=x_{0}... x_{n_{k}}}\sum_{\overline{x}=x_{n_{k}}+1 ... x_{n_{k+1}}}e^{\psi_{1}^{n_{k}+1}
(\underline{x}\overline{x}\underline{w})}e^{\psi_{1}^{n_{k+1}-n_{k}}(\overline{x}\underline{w})}}{\sum_{\underline{x}'=
x_{1}... x_{n_{k}}}\sum_{\overline{x}=x_{n_{k}}+1 ... x_{n_{k+1}}}e^{\psi_{1}^{n_{k}}
(\underline{x}\overline{x}\underline{w})}e^{\psi_{1}^{n_{k+1}-n_{k}}(\overline{x}\underline{w})}}\\
&\leq&\max_{\overline{x}}\frac{\sum_{\underline{x}=x_{0}...x_{n_{k}}}e^{\psi_{1}^{n_{k}+1}
(\underline{x}\overline{x}\underline{w})}}{\sum_{\underline{x}'=x_{1}...x_{n_{k}}}
e^{\psi_{1}^{n_{k}}(\underline{x}'\overline{x}\underline{w})}}
%&\leq&\max_{\overline{x}}u_{\overline{x}\underline{w},k}(\underline{z})
 \leq \max_{\underline{w}'}u_{\underline{w}',k}(\underline{z})
 \end{eqnarray*}

 On the other hand, for given $\ov x$
 \begin{eqnarray*}
\min_{\underline{w}'}u_{\underline{w}',k}(\underline{z})
&\leq&\min_{\overline{x}\underline{w}}
u_{\overline{x}\underline{w},k}(\underline{z})\\
&=& \min_{\overline{x}\underline{w}}\frac{\sum_{\underline{x}=x_{0}...x_{n_{k}}}e^{\psi_{1}^{n_{k}+1}
(\underline{x}\overline{x}\underline{w})}}{\sum_{\underline{x}'=x_{1}...x_{n_{k}}}
e^{\psi_{1}^{n_{k}}(\underline{x}'\overline{x}\underline{w})}}\\
&\leq&\frac{\sum_{\underline{x}=x_{0}...x_{n_{k}}}\sum_{\overline{x}=x_{n_{k}}+1 ... x_{n_{k+1}}}e^{\psi_{1}^{n_{k}+1}
(\underline{x}\overline{x}\underline{w})}e^{\psi_{1}^{n_{k+1}-n_{k}}(\overline{x}\underline{w})}}{\sum_{\underline{x}'=
x_{1}...x_{n_{k}}}\sum_{\overline{x}=x_{n_{k}}+1 ... x_{n_{k+1}}}e^{\psi_{1}^{n_{k}}
(\underline{x}\overline{x}\underline{w})}e^{\psi_{1}^{n_{k+1}-n_{k}}(\overline{x}\underline{w})}}
%&=&\frac{\sum_{\underline{x}=x_{0}...x_{n_{k+1}}}e^{\psi_{1}^{n_{k+1}+1}
%(\underline{x}\underline{w})}}{\sum_{\underline{x'}=x_{1}...x_{n_{k+1}}}
%e^{\psi_{1}^{n_{k+1}}(\underline{x}'\underline{w})}}\\
= u_{\underline{w},k+1}(\underline{z})
\end{eqnarray*}
 
\end{proof}

\begin{lemma}\label{fs.lema 1.2}
For $\nu$-a.e. $\underline{z}\in \Sigma_{2}$, the sequence $\lambda_{0}(\underline{z})\geq \lambda_{1}(\underline{z})\geq...\geq 
 \lambda_{n}(\underline{z})\geq\lambda_{n+1}(\underline{z})\geq...\geq 1$ is a decreasing sequence.
\end{lemma}
\begin{proof}
From inequalities of Lemma \ref{fs. nested}, to $\underline{w},\underline{w}'\in \Sigma_{1}$, we have
 
\[\frac{u_{\underline{w},k+1}(\underline{z})}{u_{\underline{w}',k+1}(\underline{z})}\leq
\sup_{\underline{v},\underline{v}' \in \Sigma_{1}}
\left\{\frac{u_{\underline{v},k}(\underline{z})}{u_{\underline{v}',k}(\underline{z})}\right\}= \lambda_{k}(\underline{z})\]
Taking the supremum over $\underline{w},\underline{w}'\in \Sigma_{1}$,
we have
$\lambda_{k+1}(\underline{z})\leq \lambda_{k}(\underline{z})$.
 
\end{proof}

Now, we show that $\lambda_{k}(\underline{z})\rightarrow 1$ for $\nu$-a.e. $\underline{z}\in \Sigma_{2}$. Let $n_{i}<n_{k}$ be Gibbs times of $\ud z$  and words $x_{0},...\,,x_{n_{i}}$ that projects on $\underline{z}$. We define

 \[P^{k,\,i}(\overline{x},\underline{w})=\frac{\sum_{\underline{x}=x_{1}...x_{n_{i}}}
 e^{\psi_{1}^{n_{k}}(\underline{x}\overline{x}\underline{w})}}{\sum_{\underline{x}'=x_{1}...x_{n_{k}}}
 e^{\psi_{1}^{n_{k}}(\underline{x}'\underline{w})}}\]

The probability vector $P^{k,\,i}(\overline{x},\underline{w})$ allow us to express the
function $u_{\underline{w},k}$ in terms of the function $u_{\underline{w},i}$, for $i<k$. Indeed, we have the

\begin{lemma}\label{fs.lema3.7}

 Let  $n_{i}<n_{k}$ be Gibbs times  of $\underline{z} \in B$. Then, we have that
 
\begin{eqnarray*}
u_{\underline{w},k}(\underline{z})= \sum_{\overline{x}=x_{n_{i}+1}...\,x_{n_{k}}}u_{\overline{x}
 \underline{w},i}(\underline{z})P^{k,i}(\overline{x},\underline{w}).
 \end{eqnarray*}
 where the sum above is over words $x_{0}...\,x_{n_{k}}$ that project onto $z_{0}...\,z_{n_{k}}$.
\end{lemma}

\begin{proof}

By definition, the numerator of $u_{\underline{w},k}(\underline{z})$ is

\begin{eqnarray}\label{fs. eq 1}
 \sum_{\underline{x}=x_{0}...\,x_{n_{k}}}e^{\psi_{1}^{n_{k}+1}(\underline{x}\underline{w})}=\sum_{\underline{x}=x_{0}...\,x_{n_{i}}}
\sum_{\overline{x}=x_{n_{i}+1}...\,x_{n_k}}e^{\psi_{1}^{n_{i}+1}(\underline{x}\overline{x}\underline{w})}
e^{\psi_{1}^{n_{k}-n_{i}}(\overline{x}\underline{w})} 
\end{eqnarray}

Further,  we can rewrite the right hand side of Equation \eqref{fs. eq 1} as

\begin{eqnarray}\label{fs. eq 2}
 =\sum_{\overline{x}=x_{n_{i}+1}...\,x_{n_{k}}}\underbrace{\left(\frac{\sum_{\underline{x}=x_{0}...x_{n_{i}}}e^{\psi_{1}^{n_{i}+1}(\underline{x}\overline{x}
 \underline{w})}}{\sum_{\underline{x}'=x_{1}...\,x_{n_{i}}}e^{\psi_{1}^{n_{i}}(\underline{x}'\overline{x}
\underline{w})}}\right)}_{u_{\overline{x}\underline{w},i}(\underline{z})}
 \underbrace{\left(\sum_{\underline{x}'=x_{1}...\,x_{n_{i}}}e^{\psi_{1}^{n_{i}}(\underline{x}'\overline{x}\underline{w})}\right) 
 e^{\psi_{1}^{n_{k}-n_{i}}(\overline{x}\underline{w})}}_{\sum_{\underline{x}'=x_{1}...\, x_{n_{i}}}e^{\psi_{1}^{n_{k}}(\underline{x}'
 \overline{x}\underline{w})}}\\ \nonumber
\end{eqnarray}
 
 Then, dividing both members of Equation \eqref{fs. eq 2} by
 $\sum_{\underline{x}=x_{1}...\,x_{n_{k}}}e^{\psi_{1}^{n_{k}}}(\underline{x}\,\underline{w})$,
 we have
 \[u_{\underline{w},k}(\underline{z})= \sum_{\overline{x}=x_{n_{i}+1}...\,x_{n_{k}}}u_{\overline{x}\underline{w},i}(\underline{z})
 \cdot P^{k,i}(\overline{x},\underline{w})\]
\end{proof}

\begin{corollary}\label{fs.cor3.8}
 \[\frac{u_{\underline{w},k}(\underline{z})}{u_{\underline{w}',k}(\underline{z})}=
 \frac{\sum_{\overline{x}=x_{n_{i}+1}...\,x_{n_{k}}}u_{\overline{x}
 \underline{w},i}(\underline{z})P^{k,i}(\overline{x},\underline{w})}{\sum_{\overline{x}=x_{n_{i}+1}
 ...\,x_{n_{k}}}u_{\overline{x}
 \underline{w}',i}(\underline{z})P^{k,i}(\overline{x},\underline{w}')}\]
\end{corollary}
\begin{proof}
Follows directly from Lemma \ref{fs.lema3.7}.

\end{proof}

\begin{lemma}\label{fs.lema 4.5}
There exist $c>0$ such that for any $n_{i}<n_{k}$ Gibbs times of $\underline{z}$ and for $ \overline{x}=x_{n_{i}+1}...x_{n_{k}}$ that projects onto  $ \overline{z}=z_{n_{i}+1}...z_{n_{k}}$ and 
 $\underline{w},\underline{w}'$ we have
 
 \[\frac{P^{k,i}(\overline{x},\underline{w})}{P^{k,i}(\overline{x},\underline{w}')}\geq c\]
\end{lemma} 
\begin{proof}

Since $\underline{x}\overline{x}\underline{w}$ and $\underline{x}\overline{x}\underline{w}'$ agree in $n_{k}$ places by the Gibbs Property
\eqref{eq.1gibbs} we have
$\frac{e^{\psi_{1}^{n_{k}}(\underline{x}\overline{x}\underline{w})}}{e^{\psi_{1}^{n_{k}}(\underline{x}\overline{x}\underline{w}')}}\leq
K^2$.

We can write

\begin{eqnarray*}
\frac{P^{k,i}(\overline{x},\underline{w})}{P^{k,i}(\overline{x},\underline{w'})}&=&
\frac{\sum_{\underline{x}=x_{1}...x_{n_{i}}}e^{\psi_{1}^{n_{k}}(\underline{x}\overline{x}\underline{w})}}
{\sum_{\underline{x}=x_{1}...x_{n_{i}}}e^{\psi_{1}^{n_{k}}(\underline{x}\overline{x}\underline{w'})}}
\cdot\frac{\sum_{\underline{x}'=x_{1}...x_{n_{k}}}e^{\psi_{1}^{n_{k}}(\underline{x}'\underline{w'})}}
{\sum_{\underline{x}'=x_{1}...x_{n_{k}}}e^{\psi_{1}^{n_{k}}(\underline{x}'\underline{w})}}
\leq K^4
\end{eqnarray*}

To finish the proof of Lemma \ref{fs.lema 4.5}, just take
$c=\frac{1}{K^{4}}$.
\end{proof}

\begin{lemma}\label{fs.pglema12}
With the same notations of Lemma \ref{fs.lema 4.5} we have

 \[\frac{u_{\ov x\ud w,i}(\ud z)}{u_{\ov x\ud w',i}(\ud z)}\leq e^{2\sum_{n=n_{k}-n_{i}}^{n_{k}}
 var_{n}(\psi_{1},\Pi^{-1}(\sigma^{n_k-n}(\ud z)))}\]
\end{lemma}

\begin{proof}
Considering first the numerators, we have

\[\frac{\mbox{numerator}(u_{\ov x\ud w,i}(\ud z))}{\mbox{numerator}(u_{\ov x\ud w',i}(\ud z))}=
\frac{\sum_{\ud x=x_0... x_{n_i}}e^{\psi_1^{n_i+1}(\ud x\ov x w)}}{\sum_{\ud x=x_0... x_{n_i}}e^{\psi_1^{n_i+1}(\ud x\ov x w')}}\]
comparing termwise we see that $\sigma^{j}(\ud x\ov x w)$ and $\sigma^{j}(\ud x\ov x w')$ agree to $n_k-n_i+(n_i-j)$ places, and
thus for any choice of $\ov x$,

\[\frac{e^{\psi_1^{n_i+1}(\ud x\ov x w)}}{e^{\psi_1^{n_i+1}(\ud x\ov x w')}}\leq 
e^{\sum_{n=n_{k}-n_{i}}^{n_{k}}var_{n}(\psi_{1},\sigma^{n_k-n}(\ud x))}
\leq e^{\sum_{n=n_{k}-n_{i}}^{n_{k}}var_{n}(\psi_{1},\Pi^{-1}(\sigma^{n_k-n}(\ud z)))}\]
Summing over all choices of $\ov x$ and making the identical calculations for the denominator the
lemma is proved.
\end{proof}

\begin{corollary}
  \[\frac{u_{\ov x\ud w^{max},i}(\ud z)}{u_{\ov x\ud w^{min},i}(\ud z)}\leq 
  e^{2\sum_{n=n_{k}-n_{i}}^{n_{k}}var_{n}(\psi_{1},\Pi^{-1}(\sigma^{n_k-n}(\ud z)))}\]
where $\ud w^{max}$ and $x\ud w^{min}$ are concatenation $\ov x\ud w$ which maximizes and minimizes
$u_{\ov x\ud w,i}(\ud z)$, respectively.
\end{corollary}

\begin{lemma}\label{fs.pglema13}
Let be $(n_{i})_{i\geq 1}$ the sequence of Gibbs times of $\underline{z}$. For $k\geq i$, we have

\begin{eqnarray}\label{fs.desi 3}
 \lambda_{k}(\underline{z})\leq c\cdot e^{2\sum_{n=0}^{n_{i}}
 var_{n_k-n}(\psi_{1},\Pi^{-1}(\sigma^{n}(\ud z)))}+(1-c)\lambda_{i}(\underline{z})
\end{eqnarray}

\end{lemma}

\begin{proof}

Suppose that for $k\geq i$

\[\lambda_{i}(\underline{z})= \max_{\ov x}\left(\frac{u_{\ov x\ud w^{max},i}
(\underline{z})}{u_{\ov x\ud w^{min},i}(\underline{z})}\right)\leq 
 e^{2\sum_{n=0}^{n_{i}}
	var_{n_k-n}(\psi_{1},\Pi^{-1}(\sigma^{n}(\ud z)))}.\] 
Then, by Lemma \ref{fs.lema 1.2}, $(\lambda_k)_k$ is decreasing and we have that
 
 \begin{eqnarray*}
 \lambda_{k}(\underline{z})&\leq& \lambda_{i}(\underline{z})=c\lambda_{i}(\underline{z})+(1-c)\lambda_{i}(\underline{z})
  \leq c \cdot  e^{2\sum_{n=0}^{n_{i}}
  	var_{n_k-n}(\psi_{1},\Pi^{-1}(\sigma^{n}(\ud z)))}+(1-c)\lambda_{i}(\underline{z}).
 \end{eqnarray*}
Proving the claim in the Lemma. Now assume that
$\lambda_{i}(\underline{z})= \max_{\ov x}\left(\frac{u_{\ov x\ud w^{max},i}
(\underline{z})}{u_{\ov x\ud w^{min},i}(\underline{z})}\right)>  e^{2\sum_{n=0}^{n_{i}}
var_{n_k-n}(\psi_{1},\Pi^{-1}(\sigma^{n}(\ud z)))}$.

By Corollary \ref{fs.cor3.8} we have

\begin{eqnarray}\label{fs eq3}
\frac{u_{\underline{w}^{max},k}(\underline{z})}{u_{\underline{w}^{min},k}(\underline{z})}&=&\frac{\sum_{\overline{x}=x_{n_{i}+1}...x_{n_{k}}}u_{\overline{x}
 \underline{w}^{max},i}(\underline{z})P^{k,i}(\overline{x},\underline{w}^{max})}{\sum_{\overline{x}=x_{n_{i}+1} \nonumber
 ...x_{n_{k}}}u_{\overline{x}\underline{w}^{min},i}(\underline{z})P^{k,i}(\overline{x},\underline{w}^{min})}
 \end{eqnarray}
 
To simplify notation, we will fix $\underline{z}$ and enumerate the set $\ov X$ of all possible choices $\overline{x} \in \Pi^{-1}(\ud z)$.  So, the sum $\sum_{\overline{x}=x_{n_{i}+1}...x_{n_{k}}}$ can be denoted by a sum $\sum_{l \in \ov X}$. 
  
Let $P_{1}$ be the probability vector $P^{k,i}(\overline{x},\underline{w}^{max})$ and $P_{2}$ the vector 
 $P^{k,i}(\overline{x},\underline{w}^{min})$. Also denote by $A$  the vector $(u_{\overline{x}
 \underline{w}^{\max},i}(\underline{z}))$ and $B$ the vector $(u_{\overline{x}\underline{w}^{\min},i}(\underline{z}))$, where   $\overline{x}$ run over
all choices in $\ov X$. If $a_{l}$ and $b_{l}$ represent the $l$-th term of $A$ and $B$, 
 respectively, and $I$ is vector of length $|\ov X|$ with $1$ in all of its coordinates,  we can summarize the above equality as
 
 \begin{eqnarray}\label{eq.probvectors}
  \frac{u_{\underline{w}^{max},k}(\underline{z})}{u_{\underline{w}^{min},k}(\underline{z})}&=&
  \frac{P_1 \cdot A}{P_2\cdot B}=\frac{cP_1\cdot A+(1-c)P_1 \cdot A}{c P_1\cdot B+(P_2-cP_1)\cdot B}
 \end{eqnarray}
where the signal $\cdot$ is represent the inner product in $\mathbb{R}^{|\ov X|}$.

 By Lemma \ref{fs.pglema12}, $a_l\leq b_l \, e^{2\sum_{n=0}^{n_{i}}
 	var_{n_k-n}(\psi_{1},\Pi^{-1}(\sigma^{n}(\ud z)))}$ for each $l$. Thus, 

 $$
P_1\cdot A\leq  e^{2\sum_{n=0}^{n_{i}}
	var_{n_k-n}(\psi_{1},\Pi^{-1}(\sigma^{n}(\ud z)))} P_1\cdot B.
$$ Therefore, in  the Equation \ref{eq.probvectors}  we have
 
  \begin{eqnarray*}
  \frac{u_{\underline{w}^{max},k}(\underline{z})}{u_{\underline{w}^{min},k}(\underline{z})}&\leq&\frac{c  e^{2\sum_{n=0}^{n_{i}}
  		var_{n_k-n}(\psi_{1},\Pi^{-1}(\sigma^{n}(\ud z)))}P_1\cdot B+(1-c)P_1 \cdot A}{c P_1\cdot B+(P_2-cP_1)\cdot B}\\
  &\leq& \frac{c  e^{2\sum_{n=0}^{n_{i}}
  		var_{n_k-n}(\psi_{1},\Pi^{-1}(\sigma^{n}(\ud z)))}P_1\cdot B+(1-c)P_1 \cdot I \max_l a_l}{c P_1\cdot B+(P_2-cP_1)\cdot I \min_l b_l}
 \end{eqnarray*}
 
We prove the following lemma
\begin{lemma}\label{fslem5}
 Putting $\alpha_1=c  e^{2\sum_{n=0}^{n_{i}}
 	var_{n_k-n}(\psi_{1},\Pi^{-1}(\sigma^{n}(\ud z)))}P_1\cdot B$, $\beta_1=cP_1\cdot B$,
 $\alpha_2=(1-c)P_1\cdot I \max_l a_l$ and $\beta_2=(P_2-cP_1)\cdot I \min_l b_l$. Then $\frac{\alpha_1}{\beta_1}<\frac{\alpha_2}{\beta_2}$.
\end{lemma}
\begin{proof}
	
\begin{eqnarray*}
 \alpha_2 \beta_1=cP_1\cdot B(1-c)P_1\cdot I \max_l a_l&\geq& c  e^{2\sum_{n=0}^{n_{i}}
 	var_{n_k-n}(\psi_{1},\Pi^{-1}(\sigma^{n}(\ud z)))}
 P_1\cdot B(1-c)P_1\cdot I \min_l b_l\\
 &\geq& c  e^{2\sum_{n=0}^{n_{i}}
 	var_{n_k-n}(\psi_{1},\Pi^{-1}(\sigma^{n}(\ud z)))} P_1\cdot B(P_2-cP_1)\cdot I \min_l b_l=
 \alpha_1 \beta_2
\end{eqnarray*}

\end{proof}

It is an elementary fact that $\displaystyle\frac{x_1}{y_1}<\frac{x_2}{y_2}$ implies
 $\displaystyle\frac{cx_1+x_2}{cy_1+y_2}>\frac{x_1+x_2}{y_1+y_2}$, $\forall$ $c\in(0,1)$ and positive real numbers $x_1,x_2,
 y_1$ and $y_2$. Using Lemma~\ref{fslem5} we have that

\begin{eqnarray*}
  \frac{u_{\underline{w}^{max},k}(\underline{z})}{u_{\underline{w}^{min},k}(\underline{z})}&\leq&
  \frac{c  e^{2\sum_{n=0}^{n_{i}}
  		var_{n_k-n}(\psi_{1},\Pi^{-1}(\sigma^{n}(\ud z)))}P_1\cdot I\min_l b_l+(1-c)P_1 \cdot I\max_l a_l}
  {c P_1\cdot I \min_l b_l+(P_2-cP_1)\cdot I \min_l b_l}
   \end{eqnarray*}
As $P_1$ and $P_2$ are probability vectors, then $1=P_1 \cdot I=P_2\cdot I$. Dividing by $\min_l b_l$ we have
\begin{eqnarray*}
\frac{u_{\underline{w}^{max},k}(\underline{z})}{u_{\underline{w}^{min},k}(\underline{z})}&\leq&
\frac{c e^{2\sum_{n=0}^{n_{i}}
		var_{n_k-n}(\psi_{1},\Pi^{-1}(\sigma^{n}(\ud z)))}+ (1-c)\frac{\max_{i}(a_{i})}{\min_{i}(b_{i})}}
 {c+(1-c)}\\
 &=&c e^{2\sum_{n=0}^{n_{i}}
 	var_{n_k-n}(\psi_{1},\Pi^{-1}(\sigma^{n}(\ud z)))}+ (1-c)\frac{\max_{i}(a_{i})}{\min_{i}(b_{i})} 
\end{eqnarray*}
resulting, $\lambda_{k}(\underline{z})\leq c e^{2\sum_{n=0}^{n_{i}}
	var_{n_k-n}(\psi_{1},\Pi^{-1}(\sigma^{n}(\ud z)))}+(1- c)\lambda_{i}(\underline{z})$
as we desire.

\end{proof}

\begin{corollary}\label{fs.cor14}
In $\nu-$a.e. $\underline{z}\in \Sigma_{2}$, the sequence $(\lambda_{k}(\underline{z}))_{k}$ converges to $1$.
\end{corollary}
\begin{proof}

For $k>i$, we have $n_k-n_i=(n_k-n_{k-1})+(n_{k-1}-n_{k-2})+...+(n_{i+1}-n_i)\geq k-i$. In particular, $n_{2k}-n_k\geq k$. By Lemma \ref{fs.pglema13}, we have

  \begin{eqnarray*}
  \lambda_{2k}(\underline{z})&\leq& c\,e^{2\sum_{n=0}^{n_{k}}var_{n_{2k}-n}(\psi_{1},\Pi^{-1}(\sigma^{n}(\ud z)))}+(1- c)\,\lambda_{k}(\underline{z})\Rightarrow\\
 \lambda_{2k}(\underline{z})-\lambda_{k}(\underline{z})&\leq&c\,e^{2\sum_{n=0}^{n_{k}}var_{n_{2k}-n}(\psi_{1},\Pi^{-1}(\sigma^{n}(\ud z)))}-c\lambda_k(\ud z)
 \end{eqnarray*}

 Since $\psi_{1}$ is continuous at $z$, it follows that $e^{\sum_{n=0}^{n_{k}}var_{n_{2k}-n}(\psi_{1},\Pi^{-1}(\sigma^{n}(\ud z)))}\leq 
 e^{\sum_{n=n_{2k}-n_{k}}^{n_{2k}}var_{n}(\psi_{1})}\rightarrow 1$, when $k\rightarrow \infty$.
 As the sequence $(\lambda_k(\ud z))_k$ is decreasing and 
$\lim_{k\rightarrow \infty}\lambda_k(\underline{z})\leq 1$.
 By Lemma \ref{fs.lema 1.2}  we have that $\lim_{k\rightarrow \infty}\lambda_{k}(\underline{z})\geq1$ and, therefore,  
 $\lambda_{k}(\underline{z})\rightarrow 1$.
\end{proof}

To finish the proof of the statement of main result of this section, we prove the following lemma.
\begin{lemma}\label{f.s. cont}
 For $\nu-$a.e. $\underline{z}\in \Sigma_{2}$ the limit $u(\underline{z})=\lim_{k\rightarrow \infty}u_{\underline{w},k}(\underline{z})$ exist and  $\psi_{2}(\ud z):=\log(u(\underline{z}))$ is continuous a.e.
\end{lemma}

\begin{proof}
For the existence of the limit it  sufficient to use the fact that $\lambda_{k}(\underline{z})\rightarrow 1$ in $\nu-$a.e. $\underline{z}\in \Sigma_{2}$. 

Now, to prove the continuity almost everywhere. Let $(n_{i})_{i\geq 1}$ be the sequence of Gibbs times of $\underline{z}$ and $n$ such that $n\geq n_{k}$ (note that $k\leq n$). 
Let $\ud z'\in [z_0...z_n]$. By definition of sequential Gibbs measure we have that $\Lambda_{k}(\underline{z})=\Lambda_{k}(\underline{z}')$. And the fact
the sequence $(\Lambda_{n}(\underline{z}))_{n}$ be monotonically nested, both $u(\underline{z})$ and $u(\underline{z}')$ are in interval
$\Lambda_{k}(\underline{z})$. Therefore,
 \begin{eqnarray*}
\frac{u(\underline{z})}{u(\underline{z}')}&\leq& 
 \sup_{\ud w, \ud w' \in \Sigma_1} \left\{\frac{u_{\underline{w},k}(\underline{z})}
 {u_{\underline{w}',k}(\underline{z}')}\right\}
  =\sup_{\ud w, \ud w' \in \Sigma_1} \left\{\frac{u_{\underline{w},k}(\underline{z})}
 {u_{\underline{w}',k}(\underline{z})}\right\}=\lambda_k(\ud z)
 \end{eqnarray*}
 Then 
 $|\log(u(\ud z))-\log(u(\ud z'))|\leq \log \lambda_k(\ud z)$,
that implies $\psi_2=\log u$ is continuous at $\underline{z}\in \Sigma_{2}$. Therefore, is continuous at  $\nu$- a.e. $\ud z \in \Sigma_2$.
\end{proof}

\section{Proof of Theorem \ref{fs.teo reg1}\label{subsec. reg.}: Modulus of continuity of $\psi_2$}

%\subsubsection{Proof of Theorem \ref{fs.teo reg1}: Item $1$.}\label{fs subsection1}

  Take $k_0(\ud x)$ such that $k\geq k_0(\ud x)$ implies  $n_k(\ud x)\leq b k$. Thus, we have that for a.e. $\ud z \in  \Sigma_2$, if $\ud z' \in [z_0,\dots,z_{b k}]$
\begin{equation}\label{eq.0}
|\psi_2(\ud z)-\psi_2(\ud z')| \leq \log \lambda_k,
\end{equation}
	 where   $[x]$ denotes the greatest integer smaller or equal to $x$. It follows directly that the speed of convergence  of $\log \lambda_k$ to zero give us the modulus of continuity of $\phi_2$ at $\ud z$.

To estimate $\log \lambda_k$ we observe that  given $k>k_0$ and $l\geq 2$,  for every  $2 \leq i \leq l$ we have that $n_{ik}-n_{(i-1)k}\geq k$ and that $n_{ik}\leq b ik \leq blk$. Thus, by the integral test for series we have that:
\begin{equation}\label{eq.somateta}
\sum_{j=n_{ik}-n_{(i-1)k}}^{n_{ik}}  var_{j}(\psi_{1},\Pi^{-1}(\sigma^{n_{ik}-j}(\ud z))) \leq  \sum_{j= k}^{n_{ik}}  var_{j}(\psi_{1},\Pi^{-1}(\sigma^{n_{ik}-j}(\ud z))) \leq  f(k) n_{ik} \leq f(k)blk. 
\end{equation}

By Lemma \ref{fs.pglema13},  for $\nu$-a.e $\ud z \in \Sigma_2$, for every $k>k_0(z)$ and $i=2,\dots,l$:

$$
  \lambda_{ik}(\ud z)\leq c e^{2\sum_{j=n_{ik}-n_{(i-1)k}}^{n_{ik}}  var_{j}(\psi_{1},\Pi^{-1}(\sigma^{n_{ik}-j}(\ud z)))}
  +\alpha\lambda_{(i-1)k}(\ud z).
$$ 		

Thus, using Equation~\eqref{eq.somateta}:
		
$$
  \lambda_{ik}(\ud z)\leq c e^{2f(k)blk}
  +\alpha\lambda_{(i-1)k}(\ud z).
$$ 		
		
		Multiplying by $\alpha^{l-i}$ both sides:
	$$  
	\alpha^{l-i}\lambda_{ik}(\ud z)\leq c \alpha^{l-i} e^{2f(k)blk} + \alpha^{l-i+1}\lambda_{(i-1)k}(\ud z)
	$$
	Adding all equations as above and cancelling the respective terms, we have that: 
	\begin{equation}\label{Eq.4}
\lambda_{lk}(\ud z)\leq c \sum_{i=2}^{l} \alpha^{l-i} e^{2f(k)blk} + \alpha^{l-1}\lambda_{k}(\ud z) \leq e^{2f(k)bkl} + \alpha^{l-1}\lambda_{k}(\ud z).
\end{equation}

Take  $l=\omega_k$ above any sequence for $k\geq k_0$. Dividing by $e^{2f(k)bk\omega_k}$  and using  that $\lambda_k\rightarrow 1$,   $\log(1+x) \approx x$ for $x$ small enough we have that  for $k$ big enough that

\begin{equation}\label{eq.2}
\log \lambda_{\omega_k k} \leq \frac{1}{\alpha}(\frac{\alpha}{e^{2bf_{\ud z}(k)k}})^{\omega_k}\lambda_k+2bf_{\ud z}(k)\omega_k k \leq 
\frac{2}{\alpha}\alpha^{\omega_k} + 2bf_{\ud z}(k)\omega_k k. \end{equation}

Giving $n\in \mathbb{N}$ big enough and $\gamma>0$,  define  $\beta=1-\gamma$ and consider $k_n=[n^\gamma]$ and $\omega_n=[n^\beta]$. Thus, for every $n$ big than some $n_0$, we have that $w_nk_n\leq n$ and $\log \lambda_{\omega_n k_n} \geq \log \lambda_n$. By  Equation \eqref{eq.2}  and \eqref{eq.0} we  have that for almost every $z$ there exist $n_0(z)$ such that for  every $n>n_0(z)$:
\begin{equation}\label{eq.3}
 \log \lambda_n \leq \log \lambda_{\omega_n k_n} \leq 2 \alpha^{n^{1-\gamma}} +  4bf_{\ud z}([n^\gamma])n, 
\end{equation} as we wish to show.

\vspace{1cm}
 	Giovane Ferreira\\
 	Universidade Federal do Maranhão\\
 	Campus do Bacanga, S\~ao Lu\' is, Maranh\~ao
 	Brazil\\
 giovane.ferreira@ufma.com\\
 % Important: Do not put any empty line here.
 
 	Krerley Oliveira\\
 	Universidade Federal de Alagoas\\
 	Campus A.C. Sim\~oes, Macei\'o, Alagoas
 	Brazil\\
  krerley@gmail.com\\
 % Important: Do not put any empty line here.
 % Use \affiliationthree{} for any address positioned under \affiliationone
 % Use \affiliationfour{}  for any address positioned under \affiliationtwo
\end{document}